\documentclass[11pt,twoside, reqno]{article}

\usepackage{amsfonts}
\usepackage{amssymb}
\usepackage{amsmath}
\usepackage{amsthm}
\usepackage{latexsym}
\usepackage{mathrsfs}
\usepackage{bbm}
\usepackage{braket}

\usepackage{txfonts}
\usepackage{enumerate}
\usepackage{hyperref}
\usepackage{graphicx}
\usepackage{authblk}

\usepackage{color}
\usepackage{pgf,tikz}
\usepackage{mathrsfs}
\usetikzlibrary{arrows}
\usepackage{mdframed}
\usepackage{braket}

\allowdisplaybreaks

\def\r{\right}
\def\lf{\left}

\newcommand\newdot{{\kern.8pt\cdot\kern.8pt}}

\def\ptr {/\!/}

\def\e{{\rm e}}
\def\mathpal#1{\mathop{\mathchoice{\text{\rm #1}}%
   {\text{\rm #1}}{\text{\rm #1}}%
   {\text{\rm #1}}}\nolimits}
\newcommand\Hess{\mathpal{Hess}}
\newcommand\Ric{\mathpal{Ric}}

\newcommand\id{\mathpal{id}}

\newcommand\E{\mathbb{E}}
\newcommand\R{\mathbb{R}}

\newtheorem{theorem}{Theorem}[section]
\newtheorem{lemma}[theorem]{Lemma}
\newtheorem{corollary}[theorem]{Corollary}

\theoremstyle{definition}
\newtheorem{remark}[theorem]{Remark}

\numberwithin{equation}{section}

\begin{document}

\title{\vskip-2.1cm\bf\Large  Hessian matrix estimates of heat-type equations via Bismut-Stroock Hessian formula
  \footnotetext{\hspace{-0.35cm} 2010 {\it Mathematics Subject
      Classification}. Primary: 58J65; Secondary: 35K08.
    \endgraf {\it Keywords and phrases}.  Heat equation;  Hamilton's estimate; Hessian matrix estimate; curvature; Hessian formula
    
      \endgraf This work has been supported by
    Supported in part by the National Key R\&D Program of China (2022YFA1006000, 2020YFA0712900), NNSFC (11921001) and  Natural Science Foundation in China  (Grant No.1247011327). }}

\author[1]{Li-Juan Cheng}    \author[2]{Rui-Yu  Yang}

\affil[1]{\small School of Mathematics, Hangzhou Normal
  University,\par
  Hangzhou 311121, People's Republic of China\par
   \texttt{lijuan.cheng@hznu.edu.cn;}  \vspace{1em}}

  \affil[2]{\small Center for Applied Mathematics, Tianjin
      University,\par Tianjin 300072, People's Republic of China\par
      \texttt{yang\_ry@tju.edu.cn}\vspace{1em}}

\date{\today}
\maketitle

\begin{abstract} {\noindent   In this paper, we establish a new global Hessian matrix estimate for heat-type equations on Riemannian manifolds using a Bismut-type Hessian formula. Our results feature fully explicit coefficients as well as delay / growth rate functions. These estimates yield two key applications: a novel backward weak Harnack inequality and a precise pointwise Hessian estimate for eigenfunctions.}
\end{abstract}


\section{Introduction\label{s0}}

Let $(M, g)$ be a $d$-dimensional complete Riemannian manifold with 
\begin{equation}
    \Ric\geq -K,\quad  K\geq 0,
\end{equation}
and $u$ be a positive solution of the heat equation:
\begin{equation}\label{HE0}
\partial_tu(t,x)=\Delta u(t,x),
\end{equation}
where $\Delta$ is the usual Laplace-Beltrami operator with respect to the Levi-Civita connection of $M$. 
Hamilton  \cite{hamilton1993matrix} proved the following weakened Harnack inequality for positive solutions to \eqref{HE0}:
\[
 \frac{|\nabla u|^2}{u^2} \leq \left(2K + \frac{1}{t}\right)\log\frac{A}{u}
\]
where  \(A := \sup_{x\in M,t\geq 0} u(t,x)\).  Such estimates play a fundamental role in geometric analysis, particularly in the proof of Liouville-type theorems (e.g., classifying bounded or ancient solutions) and the analysis of singularity formation in geometric flows.
Li \cite{li2016hamilton} subsequently improved this result for solutions to the generalized heat equation
\begin{equation}\label{eq:heat_Z}
\partial_t u(t,x) = (\Delta + Z) u(t,x),
\end{equation} 
under the Bakry-Émery curvature condition $\Ric_Z := \Ric - \nabla Z^\flat \geq -K$, where $Z$ is a $C^1$ vector field and $Z^\flat$ is the dual 1-form of $Z$, given by $Z^{\flat}(v)=\langle Z, v \rangle$ for $v\in TM$. The refined estimate 
\begin{equation}\label{eq:Li_Harnack}
\frac{|\nabla u|^2}{u^2} \leq \frac{2K}{1 - \e^{-2Kt}} \log\frac{A}{u}
\end{equation}
provides a quantitative improvement over Hamilton's bound, particularly through the optimized time-dependent coefficient.

 For the second derivative of the  heat equation,   Hamilton \cite{hamilton1993matrix} established  a global lower bound for the Hessian matrices of log solutions to \eqref{HE0}  under certain curvature assumptions, which yielded matix Harnack estimates of solutions at distinct spatial points for fixed time slices. Building on the ideas in \cite{hamilton1993matrix},  Han and Zhang \cite{han2016upper}  proved the following Hessian bound:
\begin{equation}\label{eq:Hessian_estimate}
\frac{\mathrm{Hess}(u)_{ij}}{u} \leq \left(\frac{5}{t} + B\right)\left(1 + \log\frac{A}{u}\right),
\end{equation}
where $A := \sup_{x\in M,t\geq 0} u(t,x)$  as previously defined, $\mathrm{Hess}(u)_{ij}$ denotes Hessian components in an orthonormal frame $\{e_i\}$, and the constant $B$ depends on the $L^\infty$-norms of the curvature tensor $\mathrm{Rm}$ and its covariant derivative $\nabla\mathrm{Ric}$.
These results are further applied to derive a backward Harnack inequality. Furthermore, Li \cite{li2015li} 
generalized the Hessian matrix estimate to a positive solution to the following equation 
\begin{equation}\label{eq:heat_V}  
\partial_t u = \Delta u + \langle \nabla V, \nabla u \rangle,
\end{equation}  
 where $V$ is a smooth potential function. More precisely,  let $\Ric_V=\mathrm{Ric} - \mathrm{Hess}(V) $
 and  \begin{equation}\label{eq:Ric_V}  
\mathrm{Ric}_V^{m,d} := \mathrm{Ric} - \mathrm{Hess}(V) - \frac{\nabla V \otimes \nabla V}{m-d},  
\end{equation}  
where $m>d$ denotes the synthetic dimension and $d$ is the manifold dimension. 
Then, under the modified curvature bound $\mathrm{Ric}_V^{m,d} \geq -K$ with 
\[
\sup_M \left(|\mathrm{Rm}| + |\mathrm{Ric}_V| + |\nabla\mathrm{Ric}_V| + |V|^2\right) < \infty,
\]  
Li \cite{li2015li}  established the following inequality:   \begin{equation}\label{eq:Hessian_CD}  
\frac{\mathrm{Hess}(u)_{ij}}{u} \leq \left( \frac{5}{t}+B\right)\left(1 + \log\frac{A}{u}\right),  
\end{equation}  
with $A := \sup_{x\in M,t\geq 0} u(t,x)$ and constant $B = 10m^{3/2}d K_V$. Here the effective curvature scale 
\begin{equation}\label{eq:K_V}  
K_V := K_1 + K_2 + \sqrt{(K_1+K_2)K + K_2 + K_1 + \sup_M |V|^2}  
\end{equation}  
incorporates $K_1 := \sup_M (|\mathrm{Rm}| + |\mathrm{Ric}_V|)$ and $K_2 := \sup_M |\nabla\mathrm{Ric}_V|$.

In this work, we consider quantitative matrix Hessian bounds on complete manifolds $(M,g)$ without boundary  via probabilistic methods.
Let  $Z$ be a $C^1$-vector field. 
Define the operator $R(Z):T_xM \otimes T_xM\rightarrow T_xM$ by  
\[
\langle R(Z)(X,Y),U \rangle := \langle \mathrm{Rm}(Z,X)Y, U\rangle, \quad X,Y,U\in T_xM.
\]  
Define the curvature quantities:  for any orthonormal basis $\{e_i\}$ of $T_xM$, 
\begin{align*}
|R^{\sharp,\sharp}(X,Y)|(x) &:= \sqrt{\sum_{i,j} \langle \mathrm{Rm}(e_i, X)Y, e_j \rangle^2(x)}, \\
\mathbf{d}^* R(X,Y)(x) &:= \sum_{i,j} \nabla_{e_i} \mathrm{Rm}(e_j, X)Y(x).
\end{align*}  
Here, for any $(0,2)$-tensor $T$ and $X,Y\in TM$, its sharp version $T^\sharp$ is the $(1,1)$-tensor satisfying  
\(
\langle T^\sharp(X), Y\rangle = T(X,Y).
\)
Note that for any $(0,2)$-tensor $T$,  we write  $|T|\leq \ell$ for  some nonnegative function $\ell$ if 
$|T(X, Y)|\leq \ell \, |X|\cdot |Y| $ for $X,Y\in TM$. 
Our analysis requires the following global curvature conditions:  
\begin{mdframed}
Condition $(\mathbf{H})$: There exist constants $K, K_1, K_2$ such that  
\begin{align}\label{Hcondition}
&\mathrm{Ric}_Z \geq -K; \tag{H1} \\
&\sup_{x\in M} |R^{\sharp,\sharp}(\cdot,\cdot)|(x)  \leq K_1; \tag{H2} \\
&\sup_{x\in M}{ \big|(\nabla \mathrm{Ric}_Z + \mathbf{d}^*R - R(Z))^\sharp\big|(x) }\leq K_2. \tag{H3}
\end{align}
\end{mdframed}
Consider the associated heat equation:  
\begin{equation}\label{H-E}  
\partial_t u = (\Delta+Z)u.
\end{equation}  
Under Condition $(\mathbf{H})$, we establish a new global Hessian matrix estimate for a positive solution to \eqref{H-E}  via probabilistic methods:  
\begin{theorem}\label{add-mainth1}
Let \( u \) be a positive solution to \eqref{H-E}. Under Condition \((\mathbf{H})\), for any \( t > 0 \),  
\begin{align}\label{es.ham.hess.11}
&\frac{\mathrm{Hess}(u)_{ij} }{u}\leq \sqrt{\frac{8K}{3( 1-\e^{-2Kt})}\, \log\frac{A}{u}} \, \lf [ \frac{K_2}{2}t+\lf(\sqrt{2t}K_1+\sqrt{ \frac{12K}{1-\e^{-2Kt}} }\r) \lf(1+\sqrt{2\log\frac{A}{u}} \r) \r ],
\end{align}
with $A := \sup_{x\in M,t\geq 0} u(t,x)$.
\end{theorem}

\noindent
Note that compared with the results in \eqref{eq:Hessian_estimate} and  \eqref{eq:Hessian_CD}, 
this result is sharp when $f$ is a constant function, and the coefficients depending on $K, K_1, K_2$ and time decay/growth rates are fully explicit. 
Moreover, compared with \eqref{eq:Hessian_CD}, our results provide point-wise bounds for \(\frac{\text{Hess}\, u}{u}\) while relaxing the curvature-dimension condition \(\text{CD}(K,m)\) to a lower Bakry-Émery curvature bound, eliminating dimensional constraints.
In particular,  under Condition $(\mathbf{H})$, 
Theorem \ref{add-mainth1}  yields a clean bound, i.e.  for any $t>0$,
\begin{equation}
    \frac{ \mathrm{Hess}(u)_{ij}}{ {u}}\leq (\sqrt{6}+2)\Big( B+\frac{1}{t}\Big) \lf(1+\log\frac{A}{u}\r), 
\end{equation}
where 
\[B:=\frac{\sqrt{6K_2}-\sqrt{3K_2}}{12}+\frac{\sqrt{6}}{3}K_1+2(K^+\vee K_2+K^+).\]

Theorem \ref{add-mainth1} is further applied to establish a weak Harnack inequality (see Corollary \ref{hessian-eigenfunction})
and   leads to a pointwise estimate for the eigenfunctions of  operator $\Delta+Z$ as well. 
Let $(\phi,\lambda) \in \text{Eig}(\Delta +Z)$ satisfy
\[
(\Delta+Z)\phi = \lambda \phi.
\]
Then under the condition $(\bf H)$, the Hessian matrix estimate \eqref{es.ham.hess.11} yields 
\begin{equation*}
    \frac{\mathrm{Hess}(\phi)_{ij} }{\phi} \leq (2+\sqrt{2}) \lf( \frac{ K_2 }{4} \sqrt{\frac{1}{3\lambda_1} + \frac{K^+}{3\lambda_1^2}}+2K_1 \sqrt{\frac{1}{3}+\frac{K^+}{3\lambda_1}}+2( \lambda +K^+ )\r)  \lf(1+ \log\frac{\|\phi\|_{\infty}}{\phi }\r),
\end{equation*}
where $\lambda_1$ is the first eigenvalue of $\Delta +Z$.

Let us explain the main idea to reprove \eqref{eq:Li_Harnack} by stochastic approach. Inspired by \cite{Li2021}, we adopt  the Bismut-Elworthy-Li formula, 
 which offers probabilistic representations for derivatives of heat semigroups on Riemannian manifolds. 
This formula was first introduced in \cite{bismut1984large} on compact manifold and later refined in \cite{elworthy1994formulae, Tha97} with a martingale approach.   Let $x_t^x$ be a diffusion  starting at $x\in M$, whose generator is $L=\frac{1}{2}(\Delta+Z)$, and $P_t$ be the associated heat semigroup with a probabilistic representation:
\begin{equation}
    P_tf(x)=\mathbb{E}\big[f(x_t^x)\mathbbm{1}_{\{t<\eta(x)\}}\big],
\end{equation}
$\eta(x)$ is the explosion time of $x_t^x$ .  It is well known that the condition $(\mathbb{H}1)$ ensures the non-explosiveness of the process, i.e.
$\eta(x)=\infty$ and then $P_tf$ for $t\geq 0$ is the unique solution to equation \eqref{H-E} with $u(0,\cdot)=f$ (see e.g. \cite{Ikeda-Watanabe:1989, Wbook2}).  
Assume  the condition $(\mathbb{H}1)$ holds. For any  fixed $T>0$, the global Bismut-Elworthy-Li formula states that given any $k\in C_b^1([0,T])$ such that $k(0)=1$ and $k(T)=0$, and for any  $f\in \mathscr{B}_b^+(M)$ and $v\in T_xM$, 
\begin{equation}\label{s0.for.bis}
    \Braket{\nabla P_Tf, v}(x)=-\mathbb{E}\lf[f(x_T^x)\int_0^{T}\Braket{Q_s(\dot{k}(s)v), \ptr_sdB_s}\r],
\end{equation}
where $B_s$ denotes a $d$-dimensional Brownian motion on $\mathbb{R}^d$, the linear operator $Q_s: T_xM\rightarrow T_{x_s(x)}M$  represents  the damped stochastic parallel transport, it solves the following covariant equation:
\begin{equation}\label{s1.eq.dspt}
    DQ_s=-\frac{1}{2}\Ric_Z^\sharp(Q_s)\,ds,\quad  Q_0=\id,
\end{equation}
where $D:=\ptr_s\circ d \circ\ptr_s^{-1}$ and $\ptr_s:T_xM\rightarrow T_{x_s}M$ denotes stochastic parallel transport.  Then applying the Young inequality  (see Lemma \ref{lem.jen} below) to this formula,  we have that for any $\alpha>0,$
\begin{equation}\label{s1.ineq.key}
    \aligned
    \Braket{\nabla P_Tf, v}(x)&=\mathbb{E}\lf[-\alpha f(x_T^x)\frac{1}{\alpha}\int_0^T \Braket{Q_s(\dot{k}_s v), \ptr_s dB_s}\r]\\
    &\leq \inf_{\alpha>0}\lf\{ \alpha P_Tf \log \frac{A}{P_Tf}+\alpha P_Tf \log\mathbb{E}\e^{-\frac{1}{\alpha }\int_0^T\Braket{\dot{k}(s)Q_s(v), \ptr_s dB_s}} \r\}
    \endaligned
\end{equation}
with $A=\Vert f\Vert_\infty=\sup_{t\geq 0, x\in M}P_tf(x)$.
From  the condition $(\mathbb{H}1)$,  it is easy to see that  $|Q_s|\leq \e^{\frac{1}{2}Ks}$ (see  Lemma \ref{s1.lem.es.dspt}), which together with  Lemma \ref{s1.lem.expmart} below,  implies
\begin{equation}\label{s1.stt1.1}
    \aligned
       \log\mathbb{E}\Big[\e^{-\int_0^T\frac{\dot{k}(s)}{\alpha}\Braket{Q_s(v), \ptr_s dB_s}}\Big]&\leq \inf_{p> 1} \lf\{\frac{(p-1)}{p}\log\mathbb{E}\lf[\e^{\frac{p^2}{2(p-1)}\int_0^T\frac{\dot{k}(s)^2}{\alpha^2}|Q_s(v)|^2ds}\r] \r\} \\
       &\leq \frac{1}{2}\int_0^T \frac{\dot{k}(s)^2}{\alpha^2}\e^{Ks}ds.
    \endaligned
\end{equation}
Consider the exponential interpolation test function 
\begin{equation}
 k(s)=\frac{\int_s^{T} \e^{-Kr}\, dr}{\int_0^T\e^{-Kr}\, dr},
\end{equation}
and define the unit vector  $v=\frac{\nabla P_T f}{|\nabla P_T f|}$.  Write $u(t,x)= P_{2t}f(x)$ for $ t\geq 0$ with $f=u(0,\cdot)$, which solves the equation \eqref{H-E}.  We then conclude from the  inequalities \eqref{s1.stt1.1} and \eqref{s1.ineq.key} that
\begin{equation}\label{s1.stt1.3}
\aligned
     \frac{|\nabla u|}{u}&\leq \inf_{\alpha>0}\lf\{\alpha \log\frac{A}{u}+\frac{1}{2\alpha}\frac{1}{\int_0^{2t}\e^{-Kr}dr}\r\}=\sqrt{\frac{2K}{1-\e^{-2Kt}}\log\frac{A}{u}}.
\endaligned
\end{equation}
Along this idea, the development of Hessian matrix estimates naturally requires the prior establishment of second-order Bismut-type formulas.  This foundational approach is well-supported in the literatures \cite{APT2003, elworthy1994formulae, Li2021,
  StT98, St00,  Tha97, Th19}.

The paper is organized as follows. We establish Theorem \ref{add-mainth1} in Section 2 through a Bismut-type Hessian formula. Building on these foundational results, Section \ref{s4} presents two key applications: we first derive a weak Harnack inequality that demonstrates how our Hessian estimates lead to new advances in regularity theory, and then we obtain a precise pointwise Hessian estimate for eigenfunctions, illustrating the power of our main theorem in concrete settings.

\section{Hessian matrix estimates}\label{s2}

Let $(M,g)$  be a $d$-dimensional smooth Riemannian manifold. 
We first introduce some auxiliary results.
It is well-known that 
by a Gronwall argument from the covariant equation \eqref{s1.eq.dspt} for $Q_s$, we can easily get the following estimate.
\begin{lemma}\label{s1.lem.es.dspt}
If  $\Ric_Z \geq h $ for $h \in C(M)$,  then for $t>s\geq 0$,
\begin{equation}\label{s1.es.dspt}
    |Q_t\circ Q_s^{-1}|\leq \exp \lf(-\frac{1}{2}\int_s^t h(x_r)\, dr \r).
\end{equation}
\end{lemma}
\begin{proof}
We define the inverse operator  $Q_t^{-1}=\ptr_t^{-1} (Q_t\ptr_t^{-1})^{-1}$, where $(Q_t\ptr_t^{-1})^{-1}$ denotes the inverse of the linear operator
$$Q_t\ptr_t^{-1}:\ T_{X_t}M\rightarrow T_{X_t}M.$$
The existence of this inverse is guaranteed by the local boundedness of $\Ric_Z$ and the construction of $Q_t$.
 Fixed $s\geq 0$.    We first  see that $\{Q_{t}\circ Q_s^{-1}\}_{\ t\geq s}$ is the solution to the following 
covariant equation:
\begin{align*}
D \tilde {Q}_t=-\frac{1}{2} {\rm Ric}^\sharp_Z(\tilde {Q}_t)\, dt,\qquad \tilde{Q}_s=\id.
\end{align*}
Moreover, by It\^{ o}'s formula, for any $v\in T_xM$, 
\begin{align*}
d | \tilde {Q}_t(v)|^2=-{\rm Ric}^\sharp_Z(\tilde {Q}_t(v),\, \tilde {Q}_t(v) )\, dt\leq -h(x_t) | \tilde {Q}_t(v)|^2\, dt.
\end{align*}
Taking the  integral on both side from $s$ to $t$ and then  using Gronwall's inequality, we then complete the proof. 
\end{proof}

As explained in the introduction, we introduce the Young inequality. Note that the following Young inequality   is a special case of  the Gibbs inequality about entropy variation, which is essential from the convexity of the entropy, see for example \cite{feng2023quantitative}.
\begin{lemma}\label{lem.jen}
For random variables $X, Y$ with $X>0, \mathbb{E} X\neq 0,$ and $ XY , \e^Y$ are integrable, then
    \begin{equation}\label{ineq.gib}
        \mathbb{E}(XY)\leq \mathbb{E}\Big(X\log\frac{X}{\mathbb{E}X}\Big)+\mathbb{E}X\log\mathbb{E}\e^Y.
    \end{equation}
\end{lemma}

Using  the property of the exponential martingale and   Cauchy's  inequality,  it is easy to obtain the following inequality,  which will help to prove Theorem \ref{add-mainth1}.
\begin{lemma}\label{s1.lem.expmart}
   If $M_t$ is a  local martingale, then one has
    \begin{equation}\label{ineq-martingale}
        \mathbb{E}\e^{M_t}\leq \inf_{p> 1} \lf(\mathbb{E}\e^{\frac{p^2}{2(p-1)}[M]_t}\r)^{\frac{p-1}{p}}.
    \end{equation}
\end{lemma}

\subsection{Proof of Theorem \ref{add-mainth1}} \label{s3}

In this subsection,  we  prove the matrix  Hessian estimates using methodology parallel to the gradient arguments presented in the introduction. This approach requires  second-order derivative representations through 
the Bismut-Stroock Hessian  formula developed in \cite{chen2023bismut}.

To describe it, for $k\in C_b^1(\R^+)$, we define a  double damped stochastic parallel transport :
$W_t^k(\cdot, \cdot):T_xM\times T_xM \rightarrow T_{X_t}M$ by
\begin{align}\label{for.ddspt}
W_t^k(v,w)&:=Q_t\int_{0}^{t}Q_s^{-1}{\rm Rm}(/\!/_sdB_s,Q_s(k(s)v))Q_s(w) \notag\\
 &\quad  -\frac{1}{2}Q_t\int_{0}^{t}Q_s^{-1}\lf(\nabla {\rm Ric}_Z^{\sharp}+{\bf d}^*R-R(Z)\r)(Q_s(k(s)v),Q_s(w))ds, 
\end{align} 
for any $v,w \in T_xM$. 
Note that for $w,u,v \in T_xM$, 
\begin{align*}
\langle {\bf d}^* R(u,v), w  \rangle= \langle (\nabla _{w}\Ric^\sharp)(u),v \rangle -\langle (\nabla _{v}\Ric^\sharp)(w),  u \rangle.
\end{align*}
We introduce the following global Bismut-Stroock Hessian formula from the local version from \cite{chen2023bismut}.
\begin{theorem}\label{thm.glob.hess}
Under the condition $(\mathbf{H})$,  for $k\in C_b^1(\mathbb{R}^+)$ such that $k(0)=1, k(s)=0$ for $s\geq T, $ and for  all $f\in \mathscr{B}_b^+(M)$ and  $v, w\in T_xM$, 
\begin{equation}\label{for.glob.hess}
    \aligned
(\Hess P_Tf)(v, w)=&-\frac{1}{2}\mathbb{E}\lf[f(x^x_T)\int_0^T\Braket{W_s^k(v, \dot{k}(s)w), \ptr_sdB_s}\r]-\frac{1}{2}\mathbb{E}\lf[f(x^x_T)\int_0^T\Braket{W_s^k(w, \dot{k}(s)v), \ptr_sdB_s}\r]\\\
&+\mathbb{E}\lf[f(x^x_T) \int_0^T\int_0^s\Braket{Q_r(\dot{k}(r)v), \ptr_rdB_r}\Braket{Q_s(\dot{k}(s)w), \ptr_sdB_s}\r]\\
&+\mathbb{E}\lf[f(x^x_T)\int_0^T\int_0^s\Braket{Q_r(\dot{k}(r)w), \ptr_rdB_r}\Braket{Q_s(\dot{k}(s)v), \ptr_sdB_s}\r]. 
    \endaligned
\end{equation}
\end{theorem}

\begin{proof}
It has been proved in \cite{chen2023bismut, CTW24} that
\begin{equation}\label{eq.thm.3.1}
    \aligned
&(\Hess P_{T-t}f)(Q_t(k(t)v),\, Q_t(k(t)w))+\langle \nabla P_{T-t}f, W^k_t(v,k(t)w) \rangle \\ 
&-\langle \nabla P_{T-t}f , Q_t(k(t)v)\rangle\ \int_0^t \langle Q_s(\dot{k}(s)w), \ptr_s dB_s \rangle-\langle \nabla P_{T-t}f , Q_t(k(t)w)\rangle\ \int_0^t \langle Q_s(\dot{k}(s)v), \ptr_s dB_s \rangle\\
&-\frac{1}{2}P_{T-t}f(x^x_t) \int_0^t\Braket{W^{k}_s(v,\dot{k}(s)w), \ptr_sdB_s}-\frac{1}{2}P_{T-t}f(x^x_t) \int_0^t\Braket{W^{k}_s(w,\dot{k}(s)v), \ptr_sdB_s}\\
&+P_{T-t}f (x^x_t) \lf(\int_0^t\int_0^s\Braket{Q_r(\dot{k}(r)v), \ptr_rdB_r}\Braket{Q_s(\dot{k}(s)w), \ptr_sdB_s}\r)\\
&+P_{T-t}f (x^x_t) \lf(\int_0^t\int_0^s\Braket{Q_r(\dot{k}(r)w), \ptr_rdB_r}\Braket{Q_s(\dot{k}(s)v), \ptr_sdB_s}\r)
    \endaligned
\end{equation}
is a local martingale.  
According to the definitions of $W^k$ and $Q$,  it  was shown in \cite{CTW24} that $\E^x|W^k_s(v,w)|^2<\infty$ and $|Q_s|<\infty$ under the condition $(\mathbf{H})$.  By \cite[Corollary 4.3]{CTW24} and \cite[Theorem 6.1 (6.5)]{ThW98}  letting  $R$ tend to $\infty$,  we find that the functions
$|\Hess P_{\cdot}f|$ and $|\nabla P_{\cdot}f|$ are bounded on
$[\epsilon, T]\times M$ for any $\epsilon>0$. These imply that  the  local martingale in  \eqref{eq.thm.3.1}
is indeed a true martingale on $[0,T-\epsilon]$.
To derive the global result, we choose $k\in C_b^1([0,T])$ such that $k(0)=1$ and $k(t)=0$ for $t\geq T-\epsilon$. By letting  $\epsilon$ tend to $0$, we obtain the second-order Bismut type formula in \eqref{for.glob.hess} for \( P_T f \) .

\end{proof}

Within these lemmas above and the Bismut-Stroock Hessian formula, we aim to prove Theorem \ref{add-mainth1}.
From the  formula \eqref{for.glob.hess}, it suffices for us to  give estimates about the  following two terms:
\begin{equation}\label{term1}
-\mathbb{E}\lf[f(x_T^x)\int_0^{T}\Braket{W_s^k(v, \dot{k}(s)w), \ptr_s \, dB_s}\r], 
\end{equation}
and
\begin{equation}\label{term2}
\mathbb{E}\lf[f(x_T^x)\int_0^T\int_0^s\Braket{Q_r(\dot{k}(r)v), \ptr_rdB_r}\Braket{Q_s(\dot{k}(s)w), \ptr_sdB_s}\r].
\end{equation}
To simplify the discussion,  we introduce the following notations :
\begin{equation}
    H=\int_0^T\e^{Ks}\dot{k}(s)^2\,ds, \quad G=\int_0^T \e^{Ks}k(s)^2\, ds, 
\end{equation}
\begin{lemma}\label{lem.s3.2}
Under the condition $(\mathbf{H})$,  for  any $f\in \mathscr{B}^+_b(M)$ and any $T>0$,
\begin{equation*}
    -\mathbb{E}\lf[ f(x_T^x)\int_0^{T}\Braket{W_s^k(v, \dot{k}(s)w), \ptr_s \, dB_s}\r]\leq \sqrt{HG} u \lf[ \lf( \frac{1}{2} K_2\sqrt{T}+ 2K_1 \r) \sqrt{2\log \frac{A}{P_Tf}} +4K_1 \log \frac{A}{P_Tf} \r],
\end{equation*}
where $A:=\|f\|_{\infty}$,   $k\in C^1_b(\R^+)$ and $v, w\in T_xM$. 
\end{lemma}
\begin{proof}
According to the definition of $W^k$,  we know that
\begin{align*}
&-\mathbb{E}\lf[ f(x_T^x)\int_0^{T}\Braket{W_s^k(v, \dot{k}(s)w), \ptr_s \, dB_s}\r]\\
&=-\mathbb{E}\lf[ f(x_T^x) \int_0^T \dot{k}(t) \Braket {Q_t\int_{0}^{t}Q_s^{-1}{\rm Rm}(/\!/_sdB_s,Q_s(k(s)v))Q_s(w), \ptr _t \, d B_t }\r]\\
&\quad +\frac{1}{2}\mathbb{E} \lf[f(x_T^x) \int_0^T \dot{k}(t) \Braket{Q_t\int_{0}^{t}Q_s^{-1}\lf(\nabla {\rm Ric}_Z^{\sharp}+{\bf d}^*R-R(Z)\r)(Q_s(k(s)v),Q_s(w))\,ds, \ptr_t\, d B_t}\r]\\
&:={\rm I}+{\rm II}.
\end{align*}
We firstly deal with the easy part {\rm II}.  Using Young's  inequality \eqref{ineq.gib}, we know that for $\alpha>0$, 
\begin{align*}
&\mathbb{E} \lf[f(x_T^x) \int_0^T \dot{k}(t) \Braket{Q_t\int_{0}^{t}Q_s^{-1}\lf(\nabla {\rm Ric}_Z^{\sharp}+{\bf d}^*R-R(Z)\r)(Q_s(k(s)v),Q_s(w))\,ds, \ptr_t\, d B_t}\r]\\
&\leq  \mathbb{E} \lf[ \alpha f (x^x_T) \log \frac{f(x_T^x)}{\mathbb{E}f(x_T^x)} \r] \\
&\quad +\mathbb{E}[\alpha f(x_T^x)]\log\mathbb{E}\e^{-\frac{1}{\alpha} \int_0^T \dot{k}(t) \Braket{Q_t\int_{0}^{t}Q_s^{-1}\big(\nabla {\rm Ric}_Z^{\sharp}+{\bf d}^*R-R(Z)\big)(Q_s(k(s)v),Q_s(w))\,ds, \ptr_t\, d B_t}}.
\end{align*}
Since $ \Ric_Z \geq -K$, by Lemma \ref{s1.lem.es.dspt}, it is easy to know that $|Q_t\circ Q_s^{-1}|\leq \e^{\frac{1}{2}K(t-s)}$ for $t\geq s$,  which, together with the condition ($\mathbb{H}$3) and the inequality \eqref{ineq-martingale} for $p$ tending to $1$,  further imply
\begin{align*}
&\mathbb{E}\e^{-\frac{1}{\alpha} \int_0^T \dot{k}(t) \Braket{Q_t\int_{0}^{t}Q_s^{-1}\big(\nabla {\rm Ric}_Z^{\sharp}+{\bf d}^*R-R(Z)\big)(Q_s(k(s)v),Q_s(w))\,ds, \ptr_s\, d B_t}}\\
&\leq \exp\lf\{\frac{K_2^2}{2\alpha^2}  \int_0^T\dot{k}(t)^2\e^{Kt} \lf(\int_0^t \e^{\frac{1}{2}Ks}k(s) \, ds\r)^2\, dt \r\}.
\end{align*}
Therefore, we conclude that
\begin{align*}
&\mathbb{E} \lf[f(x_T^x) \int_0^T \dot{k}(t) \Braket{Q_t\int_{0}^{t}Q_s^{-1}\lf(\nabla {\rm Ric}_Z^{\sharp}+{\bf d}^*R-R(Z)\r)(Q_s(k(s)v),Q_s(w))\,ds, \ptr_s\, d B_t}\r]\\
&\leq P_Tf \inf_{\alpha>0}\lf\{ \alpha  \log \frac{A}{P_Tf}+\frac{K_2^2}{2\alpha}  \int_0^T\dot{k}(t)^2\e^{Kt} \lf(\int_0^t \e^{\frac{1}{2}Ks}k(s) \, ds\r)^2\, dt \r\} \\
&=  K_2 P_Tf \sqrt{\log \frac{A}{P_Tf}} \sqrt{ 2\int_0^T\dot{k}(t)^2\e^{Kt} \lf(\int_0^t \e^{\frac{1}{2}Ks} k(s)\, ds\r)^2\, dt}\\
&\leq   K_2 \sqrt{2H} \lf(\int_0^T \e^{\frac{1}{2}Ks} k(s)\, ds \r) P_Tf \sqrt{\log \frac{A}{P_Tf}}\\
&\leq  K_2 \sqrt{H G T}P_Tf \sqrt{2\log \frac{A}{P_Tf}}.
\end{align*}
 We now deal with the term {\rm I}. 
Using Young's inequality \eqref{ineq.gib} again, we have that for any $\alpha>0, $
\begin{align}\label{ineq.lem.main.s3}
&-\mathbb{E}\lf[ f(x_T^x) \int_0^T \dot{k}(t)\Braket {Q_t\int_{0}^{t}Q_s^{-1}{\rm Rm}(/\!/_sdB_s,Q_s(k(s)v))Q_s(w), \ptr _t \, d B_t }\r] \notag\\
& \leq \mathbb{E} \lf[ \alpha f (x_T^x) \log \frac{f(x_T^x)}{\mathbb{E}f(x_T^x)} \r]  +\mathbb{E}[\alpha f(x_T^x)]\log\mathbb{E}\e^{-\frac{1}{\alpha}  \int_0^T \dot{k}(t)\Braket {Q_t\int_{0}^{t}Q_s^{-1}{\rm Rm}(/\!/_sdB_s,Q_s(k(s)v))Q_s(w), \ptr _t \, d B_t }}.\
\end{align}
Fixed $t\in [0,T]$,  by the inequality \eqref{ineq-martingale} for $p=2$, 
 we have the estimate 
 \begin{align}\label{esti-I}
&\mathbb{E}\e^{-\frac{1}{\alpha}  \int_0^T \dot{k}(t)\Braket {Q_t\int_{0}^{t}Q_s^{-1}{\rm Rm}(/\!/_sdB_s,Q_s(k(s)v))Q_s(w), \ptr _t \, d B_t }} \notag\\
&\qquad \leq \sqrt{\E \e ^{\frac{2}{\alpha^2} \int_0^T\dot{k}(t)^2 \lf|Q_t\int_{0}^{t}Q_s^{-1}{\rm Rm}(/\!/_sdB_s,Q_s(k(s)v))Q_s(w)\r|^2\, dt  }}\notag\\
&\qquad \leq \sqrt{\sup_{t\in [0,T]} \E \e^{\frac{2H}{\alpha^2}  \lf| \e^{-Kt/2}Q_t\int_{0}^{t}Q_s^{-1}{\rm Rm}(/\!/_sdB_s,Q_s(k(s)v))Q_s(w)\r|^2 }}.
\end{align}
Let  $Y_t=  \e^{-Kt/2}Q_t\int_{0}^{t}Q_s^{-1}{\rm Rm}(/\!/_sdB_s,Q_s(k(s)v))Q_s(w)$.  Then
using  It\^{o}'s  formula and the conditions ($\mathbb{H}1$) and ($\mathbb{H}2$),  we  have
\begin{align}
 d\lf| Y_t\r|^2 
&\leq 2 \e^{-Kt/2} \lf\langle Y_t ,  \, {\rm Rm}(/\!/_tdB_t,Q_t(k(t)v))Q_t(w) \r\rangle  +\e^{-Kt} |R^{\sharp, \sharp}(Q_t(k(t)v), Q_t(w)) |^2\, dt   \notag \\
&\leq 2 \e^{-Kt/2} \lf\langle Y_t,\,  {\rm Rm}(/\!/_tdB_t,Q_t(k(t)v))Q_t(w) \r\rangle  +\e^{Kt}K_1^2k(t)^2\, dt,
\end{align}
which  turns out that
\begin{equation}
\aligned
\sup_{t\in [0,T]}\E\e^{\frac{2H}{\alpha^2} \lf| Y_t \r|^2}
&\leq  \sup_{t\in [0,T]}\E\e^{\frac{4H}{\alpha^2}  \int _0^t\e^{-Ks/2} \lf\langle Y_s, \,  {\rm Rm}(/\!/_sdB_s,\, Q_s(k(s)v))Q_s(w) \r\rangle +\frac{2H}{\alpha^2}  K_1^2 \int_0^t \e^{Ks}k(s)^2\, ds } \\
&\leq \e^{\frac{2H}{\alpha^2}   K_1^2 G}\sqrt{\E\e^{\frac{32H^2}{\alpha^4}  K_1^2\int_0^T | Y_s|^2\e^{Ks}k(s)^2\,ds  }} \\
&\leq \e^{\frac{2H}{\alpha^2}  K_1^2 G}\sqrt{\sup_{t\in [0,T]} \E\e^{\frac{32H^2 G}{\alpha^4}  K_1^2   | Y_t|^2 }}.
\endaligned
\end{equation}
Let  $\alpha$ be chosen such that 
$$\frac{32H^2  G}{\alpha^4}  K_1^2 \leq \frac{2H}{\alpha^2} ,$$
i.e.,
$
\alpha\geq 4 K_1\sqrt{H G}.
$
Then 
\begin{equation}\label{ineq.s3.main.2}
 \sup_{t\in [0,T]}\E\e^{\frac{2H}{\alpha^2} \lf| Y_t \r|^2}=   \sup_{t\in [0,T]}\E\e^{\frac{2H}{\alpha^2} \lf|\e^{-Kt/2}Q_t\int_{0}^{t}Q_s^{-1}R(/\!/_sdB_s,Q_s(k(s)v))Q_s(w)\r|^2}
\leq \e^{\frac{4H}{\alpha^2}K_1^2 G }.
\end{equation}
By this with \eqref{esti-I} and \eqref{ineq.lem.main.s3},  we  conclude that
\begin{align*}
&-\mathbb{E}\lf[ f(x_T^x) \int_0^T \dot{k}(t)\Braket {Q_t\int_{0}^{t}Q_s^{-1}{\rm Rm}(/\!/_sdB_s,Q_s(k(s)v))Q_s(w), \ptr _t \, d B_t }\r] \\
&  \leq \inf_{\alpha \geq 4 K_1\sqrt{HG}} \lf\{ \mathbb{E} \lf[ \alpha f (x_T^x) \log \frac{f(x_T^x)}{\mathbb{E}f(x_T^x)} \r]  +\frac{2HGK_1^2}{\alpha}\mathbb{E}[\ f(x_T^x)]\r\}\\
&\leq \inf_{\alpha \geq 4 K_1\sqrt{HG}} \lf\{ (\alpha-4 K_1\sqrt{HG}) \mathbb{E} \lf[ f (x^x_T) \log \frac{f(x_T^x)}{\mathbb{E}f(x_T^x)} \r]  +\frac{2HGK_1^2}{\alpha-4 K_1\sqrt{HG}}\mathbb{E}[\ f(x_T^x)]\r\}\\
&\qquad \quad + 4 K_1\sqrt{HG} \E\lf[f (x_T^x) \log \frac{f(x_T^x)}{\mathbb{E}f(x_T^x)}\r]\\
&\leq  2K_1 \sqrt{HG} P_Tf \lf(2 \log \frac{A}{P_Tf} +\sqrt{2 \log \frac{A}{P_Tf}}\r) .
\end{align*}
Combining  this and the estimate result of {\rm II},  we conclude  that
\begin{equation}\label{ineq.s3.key}
\aligned
&-\mathbb{E}\lf[ f(x_T^x)\int_0^{T}\Braket{W_s^k(v, \dot{k}(s)w), \ptr_s \, dB_s}\r]\\
&\leq  \lf ( \frac{1}{2} K_2 \sqrt{2H GT}+2K_1 \sqrt{2HG}\r)P_Tf \sqrt{\log \frac{A}{P_Tf}}+ 4K_1 \sqrt{HG} P_Tf \log \frac{A}{P_Tf} \\
&=\sqrt{HG} P_Tf \lf[ \lf( \frac{1}{2} K_2\sqrt{T}+ 2K_1 \r) \sqrt{2\log \frac{A}{P_Tf}} +4K_1 \log \frac{A}{P_Tf} \r].
\endaligned
\end{equation}
\end{proof}

We 
now  study the term \eqref{term2} and present the result as follows.
\begin{lemma}\label{lem.s3.3}
Under the condition $(\mathbb{H}1)$,  for any $f\in \mathscr{B}^+_b(M)$ and any $T>0$,
\begin{equation*}
\mathbb{E}\lf[f(x_T^x)\int_0^T\int_0^s\Braket{Q_r(\dot{k}(r)v), \ptr_rdB_r}\Braket{Q_s(\dot{k}(s)w), \ptr_sdB_s}\r]\leq 2u H\lf(\sqrt{2\log\frac{A}{P_Tf}}+2\log\frac{A}{P_Tf}\r),
\end{equation*}
where $A=\|f\|_{\infty}$,  $k\in C_b^1(\R^+)$ and $v, w\in T_xM$. 
\end{lemma}
\begin{proof}
Using the Young inequality \eqref{ineq.gib}, we have for any $\alpha>0, $
\begin{align}\label{ineq.lem.main.s3.1}
& \mathbb{E}\lf[f(x_T^x)\int_0^T\int_0^s\Braket{Q_r(\dot{k}(r)v), \ptr_rdB_r}\Braket{Q_s(\dot{k}(s)w), \ptr_sdB_s}\r] \notag \\
& \leq \mathbb{E} \lf[ \alpha f (x_T^x) \log \frac{f(x_T^x)}{\mathbb{E}f(x_T^x)} \r] +\mathbb{E}[\alpha f(x_T^x)]\log\mathbb{E}\e^{\frac{1}{\alpha}\int_0^TX_s\Braket{Q_s(\dot{k}(s)w), \ptr_sdB_s}},
\end{align}
where $X_s=\int_0^s\Braket{Q_r(\dot{k}(r)v), \ptr_r dB_r}$. 
By Lemma \ref{s1.lem.expmart} for $p=2$ one has
\begin{equation}\label{ineq.s3.main.7}
    \aligned
\mathbb{E}\e^{\frac{1}{\alpha}\int_0^TX_s\Braket{Q_s(\dot{k}(s)w), \ptr_s dB_s}}&\leq \lf(\mathbb{E}\e^{\frac{2}{\alpha^2}\int_0^TX_s^2|Q_s(w)|^2\dot{k}(s)^2\,ds}\r)^{1/2}\\
&\leq \lf(\sup_{s\in [0,T]}\mathbb{E}\e^{\frac{2H}{\alpha^2} X_s^2}\r)^{1/2}.
    \endaligned
\end{equation}
Using Ito's formula for $X_s^2$ and the condition $(\mathbb{H}1)$, we find that  for any $s\in [0, T], $
\begin{equation}
\aligned
    X_s^2&\leq 2\int_0^s X_r\Braket{Q_r(\dot{k}(r)v), \ptr_rdB_r}+\int_0^T |Q_r(\dot{k}(r)v)|^2\,dr\\
    &\leq 2\int_0^s X_r\Braket{Q_r(\dot{k}(r)v), \ptr_rdB_r}+H.
    \endaligned
\end{equation}
Let $\beta =\frac{2H}{\alpha^2}$. Then 
\begin{equation}\label{term2-ineq1}
\aligned
\sup_{s\in [0,T]} \mathbb{E}\e^{\beta X_s^2}&\leq \sup_{s\in [0,T]}\mathbb{E}\e^{\beta H}\e^{2\beta \int_0^s X_r \Braket{Q_r(\dot{k}(r)v), \, \ptr_r dB_r}}\\
&\leq \e^{\beta H} \lf(\mathbb{E}\e^{8\beta^2\int_0^T X_r^2 \e^{Kr}\dot{k}(r)^2\, dr}\r)^{1/2}\\
&\leq \e^{\beta H}\lf(\sup_{s\in [0,T]}\mathbb{E}\e^{8\beta^2HX_s^2}\r)^{1/2}.
\endaligned
\end{equation}
Let $\beta \geq 8\beta^2 H, $ i.e.
$\alpha\geq 4H.$
Then
\begin{equation}\label{term2-ineq2}
  \sup_{s\in [0,T]}  \mathbb{E}\e^{\beta X_s^2}\leq \e^{2\beta H}.
\end{equation}
Combining this with \eqref{ineq.s3.main.7} yields
\begin{equation}
    \aligned
\mathbb{E}\e^{\frac{1}{\alpha}\int_0^TX_s\Braket{Q_s(\dot{k}(s)v), \, \ptr_s dB_s}}\leq \e^{\beta H},
    \endaligned
\end{equation}
which, together with \eqref{term2-ineq1}, further imply 
\begin{equation*}
 \aligned
&\mathbb{E}\lf[f(x_T^x)\int_0^T\int_0^s\Braket{Q_r(\dot{k}(r)v), \ptr_rdB_r}\Braket{Q_s(\dot{k}(s)w), \ptr_sdB_s}\r] \\
&\leq \inf_{\alpha\geq 4H} \lf \{\alpha P_Tf\log\frac{A}{P_Tf}+\frac{2H^2P_Tf}{\alpha}\r\}\\
&=\inf_{\alpha> 4H} \lf\{(\alpha-4H)P_Tf\log\frac{A}{P_Tf}+4Hu\log\frac{A}{P_Tf}+\frac{2H^2P_Tf}{\alpha-4H} \r\} \\
&=2Hu\sqrt{2\log\frac{A}{P_Tf}}+4Hu\log\frac{A}{P_Tf}.
\endaligned
\end{equation*}
We then end the proof.
\end{proof}

\begin{proof}[Proof of Theorem \ref{add-mainth1}]
Let $v=e_i, w=e_j$ where $\{e_i\}$ be  the orthonormal frame of $T_xM$. 
Combining the Bismut Hessian formula \eqref{for.glob.hess} with Lemmas \ref{lem.s3.2} and \ref{lem.s3.3}, we derive the following estimate for the Hessian:  
\begin{equation}\label{ineq.s3.main.3}
\frac{\Hess (P_Tf)_{ij}}{u}\leq   \sqrt{H}  \lf[ \lf(\frac{K_2}{2}\sqrt{2G T} +2K_1\sqrt{2G} +4\sqrt{2H} \r) \sqrt{\log \frac{A}{P_Tf}} +4(K_1\sqrt{G}+2\sqrt{H} )\log \frac{A}{P_Tf} \r].\end{equation}
Define the function \( k(s) = \frac{\int_s^T \e^{-Kr}\,dr}{\int_0^T \e^{-Kr}\,dr} \) for \( 0 \leq s \leq T \). By direct computation, we obtain:  
\begin{align}
    G&= \frac{\int_0^T \e^{Ks} (\int_s^T\e^{-Kr}\, dr)^2\, ds}{(\int_0^T \e^{-Kr}\, dr)^2}=\frac{ \e^{2TK}-1-2TK \e^{TK} }{  K(\e^{KT}-1)^2} \leq \frac{T}{3}, \label{inequality G} \\
      H&=\frac{-K}{\e^{-KT}-1}.\notag
\end{align}
(The proof of the inequality \eqref{inequality G} is deferred to the Appendix.)   
Observing that $u(t,x)=P_{2t}u(0,\cdot)(x)$ and substituting these bounds into \eqref{ineq.s3.main.3} with $T=2t$. Then we  clarify the upper bound as follows:  
\begin{align}\label{ineq.s3.main.6}
\frac{\Hess (u)_{ij}}{u} &\leq \sqrt{\frac{K}{3( 1-\e^{-2Kt})}} \lf [ K_2t \sqrt{2\log\frac{A}{u}} +2\lf(K_1\sqrt{2t}+ 2\sqrt{ \frac{3K}{1-\e^{-2Kt}}} \r) \lf(2\log\frac{A}{u} +\sqrt{2\log\frac{A}{u} }\r)\r], \end{align}
where $A=\sup_{x\in M}u(0,x)=\sup_{t\geq 0,x\in M} u(t,x)$.
\end{proof}

We now establish an improved Hessian estimate for solutions to heat-type equations, refining the previous result (\ref{eq:Hessian_estimate}) in \cite{han2016upper}. 
 \begin{corollary}\label{s2.cor.1}
Let  ${u}$ be a positive solution to  the heat equation 
\begin{equation}\label{eq.heat.2}
    \partial_t {u}=(\Delta+Z) {u}.
\end{equation}
Under the condition  $(\mathbf{H})$,   we have that for any $t>0$, 
\begin{equation}
    \frac{ \Hess (u)_{ij}}{ {u}}\leq (\sqrt{6}+2)\Big( B+\frac{1}{t}\Big) \lf(1+\log\frac{A}{u}\r),
\end{equation}
where $A=\sup_{[0,\infty)\times M} u$ and 
\[B:=\frac{\sqrt{6K_2}-\sqrt{3K_2}}{12}+\frac{\sqrt{6}}{3}K_1+2(K^+\vee K_2+K^+),\]
with $K^+ = \max\{K,0\}$.
\end{corollary}
\begin{proof}
We first observe the elementary inequality for $t>0$:
\begin{align}\label{est-eK}
\frac{K}{1-\e^{-2Kt}} = K + \frac{K}{\e^{2Kt}-1} \leq \frac{1}{2t} + K^+.
\end{align}
Using this estimate to bound \eqref{ineq.s3.main.6} and recalling that the semigroup satisfies the composition property $u = P_{2t}u(0,\cdot)= P_{(2t)\wedge c} \circ P_{2t-(2t)\wedge c}u(0,\cdot)$,
we  achieve the following upper bound by replacing   $2t$ by $(2t)\wedge c$ in \eqref{ineq.s3.main.6}.
\begin{equation}\label{ineq-c}
\aligned
\frac{\Hess (u)_{ij}}{u}&\leq  K_2\sqrt{\frac{c+K^+c^2}{6}\log\frac{A}{u}}+4\lf(\frac{K_1}{2}\sqrt{\frac{1+K^+c}{3}}+\frac{1}{(2t)\wedge c}+K^+\r) \lf(\sqrt{2\log\frac{A}{u}}+2\log\frac{A}{u}\r)
\endaligned
\end{equation}
for any  $c>0$, where $A=\sup_{[0,\infty)\times M} u$ and  $K^+ = \max\{K,0\}$. 
Applying  this inequality with parameters \( c = \frac{1}{ K^+ \vee K_2 }\),  we obtain 
  \begin{equation}\label{version-Hessian3}
\aligned
\frac{\Hess (u)_{ij}}{u}
&\leq  \sqrt{ \frac{K_2}{6} \log \frac{A}{u}}+2\lf(\sqrt{\frac{2}{3}}K_1+2(K^+\vee K_2+K^+)+\frac{1}{t}\r) \lf(\sqrt{2\log\frac{A}{u}}+2\log\frac{A}{u}\r).
\endaligned
\end{equation}
To further estimate these terms, we apply the mean value inequality: for any \( \epsilon > 0 \),
\begin{align*}
\sqrt{2\log \frac{A}{u}} + 2\log \frac{A}{u} \leq \frac{\sqrt{2}}{2\epsilon} + \frac{\sqrt{2}\epsilon}{2}\log\frac{A}{u} + 2\log \frac{A}{u}.
\end{align*}
Choosing the optimal value \( \epsilon = \sqrt{3} - \sqrt{2} \) yields
\begin{align}\label{est-1}
\sqrt{2\log \frac{A}{u}} + 2\log \frac{A}{u} \leq \frac{\sqrt{6} + 2}{2}\left(1 + \log \frac{A}{u}\right).
\end{align}
Additionally, we have
\begin{align}\label{est-2}
\sqrt{\frac{K_2}{6}\log \frac{A}{u}} \leq \frac{\sqrt{6K_2}}{12}\left(1 + \log \frac{A}{u}\right).
\end{align}
By  these inequalities \eqref{est-1}, \eqref{est-2}  and  \eqref{version-Hessian3},   we then end the proof.

\end{proof}

\section{Applications}\label{s4}

As shown in \cite{han2016upper}, Hessian matrix estimates of the form \eqref{es.ham.hess.11} yield a backward Harnack inequality. This enables the temperature at a point \( x \) at a later time \( t \) to be bounded by its value at the same point \( x \) at an earlier time \( s \).    
\begin{corollary}\label{add-cor-1}
Assume that  the condition $(\mathbf{H})$ holds.
Let  $u$ be a positive  solution to the  heat equation \eqref{H-E}.
For  $t>s>0$ and $x\in M$,  one has
\begin{equation}
  u(t,x)\leq \e^{\gamma}A^{1-\eta}u(s,x)^\eta
\end{equation}
where  $A= \sup_{t>0,x\in M}u(t,x)$, 
\begin{equation}
    \gamma=2d^2\lf(\int_s^t \lf( \Big(\frac{K_2\sqrt{2r}}{4}+K_1\Big)\sqrt{\frac{2Kr}{ 3(1-\e^{-2Kr})}} +\frac{K}{1-\e^{-2Kr}}\r)\, dr\r)^2
\end{equation}
and
\begin{equation}
    \eta=\frac{1}{2}\exp\left\{-4d\int_s^t \lf(K_1\sqrt{\frac{2Kr}{ 3(1-\e^{-2Kr} )}}+ \frac{K}{1-\e^{-2Kr}}\r)\,dr\right\}.
\end{equation}
\end{corollary}
\begin{proof}
Using Theorem \ref{add-mainth1} and the inequality $\Delta u \leq d |\Hess u|$, we  first obtain
\begin{align}\label{add-eq2}
   -\partial_t\sqrt{\log\frac{A}{u}} & \leq \frac {d}{2} \sqrt{\frac{2K}{3( 1-\e^{-2Kt})}} \lf [ \frac{1}{4}K_2t+\frac{\sqrt{2t}}{2}K_1+\sqrt{ \frac{3K}{1-\e^{-2Kt}} }  +\lf(K_1\sqrt{t}+ \sqrt{ \frac{6K}{1-\e^{-2Kt}}} \r) \sqrt{ \log\frac{A}{u} }\r].
\end{align}
To simplify the analysis, we define the time-dependent coefficients:
\begin{align*}
a_t&=\sqrt{2}d\lf[\lf( \frac{K_2\sqrt{2t}}{4}+K_1\r)\sqrt{\frac{2Kt}{ 3(1-\e^{-2Kt})}} +\frac{K}{1-\e^{-2Kt}}\r], 
\\
    b_t&=2d \lf(K_1\sqrt{\frac{2Kt}{3( 1-\e^{-2Kt})}}+ \frac{ K}{1-\e^{-2Kt}} \r).
\end{align*}
An application of  Gronwall's  inequality to the differential inequality \eqref{add-eq2} yields
\begin{equation*}
    -\sqrt{\log\frac{A}{u_t}}-\int_s^t \e^{-\int_r^t b_u\, du}a_r\, dr\leq -\sqrt{\log\frac{A}{u_s}}\exp\left\{-\int_s^tb_r\,dr\right\},
\end{equation*}
which  implies 
\begin{align*}
\lf(\log \frac{A}{u_s} \r)\exp\left\{-2\int_s^tb_r\,dr\right\}\leq 2 \log \frac{A}{u_t}+2\lf( \int_s^t a_r\, dr\r)^2.
\end{align*}
We then end the proof.

\end{proof}

\begin{remark}
Because of $\frac{2Kr}{ 1-\e^{-2Kr}}\leq 2K^+r+1$, we have the following inequalities to estimate $\gamma$ and $\eta$,
\begin{align*}
 \gamma &\leq 2d^2\lf(\int_s^t \lf( \frac{\sqrt{3}}{12}\big(K_2\sqrt{ 2r }+4K_1 \big) \sqrt{ 1+2K^+ r} +\frac{K}{1-\e^{-2Kr}}\r)\, dr\r)^2\\
      &\leq 2d^2\lf(\int_s^t \lf(  \frac{\sqrt{3}}{12} \big(K_2\sqrt{ 2r }+4K_1 \big) (1+\sqrt{2K^+r}) +\frac{K}{1-\e^{-2Kr}}\r)\, dr\r)^2\\
      &= 2d^2\lf( \frac{\sqrt{6}}{18} \lf( K_2  +4K_1\sqrt{K^+} \r) (t^{\frac{3}{2}} -s^{\frac{3}{2}}) +\frac{K_2\sqrt{3K^+}}{12}(t^2-s^2)+\frac{\sqrt{3}}{3}K_1(t-s)+\frac{1}{2}\log \frac{\e^{2Kt}-1}{\e^{2Ks}-1}\r)^2\\
      &:=\tilde{\gamma},
\end{align*}
and 
\begin{align}
    \eta&\geq   \frac{1}{2}\exp\left\{-4d \int_s^t \lf(\frac{\sqrt{3}}{3}K_1 \sqrt{1+2K^+r} + \frac{K}{1-\e^{-2Kr}}\r)\,dr\right\} \notag \\
    &= \frac{1}{2} \e^{\frac{4d\sqrt{3}}{3}K_1 (s-t) + \frac{8d\sqrt{6}}{9}K_1 \sqrt{K^+} (s^{3/2}-t^{3/2})} \lf( \frac{\e^{2Ks}-1}{\e^{2Kt}-1}\r)^{2d}:=\tilde{\eta}.
\end{align}
Then from Corollary  \ref{add-cor-1}, for  $t>s>0$ and  $x\in M$,  one has
\begin{equation*}
  u(t,x)\leq \e^{\tilde{\gamma}}A^{1- \tilde{\eta}}u(s,x)^{\tilde{\eta}}.
\end{equation*}

\end{remark}

We now investigate Hessian estimates for eigenfunctions of the operator \( L \). Given an eigenvalue pair \((\phi, \lambda) \in \text{Eig}(\Delta+Z)\), the eigenfunction \(\phi\) satisfies  
\[
(\Delta +Z)\phi = \lambda \phi.
\]  
In \cite{CTW}, the uniform Hessian estimate of $\phi$ has been given. Here, we employ the Hessian matrix estimate from Theorem \ref{add-mainth1} to derive the following pointwise bound for \( \Hess ( \phi ) \).

\begin{corollary}\label{hessian-eigenfunction}
   Let $\lambda_1$ be the first eigenvalue of $L$.  Assume that the condition $(\mathbf{H})$ holds.   For $(\phi, \lambda)\in \text{ \rm Eig}(L)$ and  for any $x\in M$, 
\begin{equation}\label{point-eq}
    \frac{\Hess( \phi) _{ij} (x)}{\phi(x)} \leq (2+\sqrt{2}) \lf( \frac{ K_2 }{4} \sqrt{\frac{1}{3\lambda_1} + \frac{K^+}{3\lambda_1^2}}+2K_1 \sqrt{\frac{1}{3}+\frac{K^+}{3\lambda_1}}+2( \lambda +K^+ )\r)  \lf(1+ \log\frac{\|\phi\|_{\infty}}{\phi (x)}\r).
\end{equation}

\end{corollary}

\begin{proof}

If $f=\phi$, then $P_Tf= \e^{-2\lambda T}\phi$, due to the proof of Theorem \ref{add-mainth1}, we compute that
\begin{align*}
&\frac{\Hess( \phi) _{ij} (x) }{\phi(x)}\leq \sqrt{ \frac{2K}{ 3(1-\e^{-2KT})}\lf(2\lambda T+ \log\frac{\|\phi\|_{\infty}}{\phi(x)}\r)} \\
&\qquad \qquad  \quad \times \, \lf [ K_2T+2\lf(\sqrt{2T}K_1+\sqrt{ \frac{12 K}{1-\e^{-2KT}} }\r) \lf(1+\sqrt{2 \lf(2\lambda T+ \log\frac{\|\phi\|_{\infty}}{\phi(x)}\r)} \r) \r].
\end{align*}
Using the inequality \( \frac{K}{1 - \e^{-2KT}} \leq \frac{1}{2T} + K^+ \) and setting \( T = \frac{1}{2\lambda} \), we simplify the estimate to  
\begin{align*}
\frac{\Hess( \phi) _{ij} (x)}{\phi(x)}&\leq 2 \sqrt{\frac{2}{3}\lf( \lambda + K^+\r)\lf(1+ \log\frac{\|\phi\|_{\infty}}{\phi(x)}\r)} \\
&\qquad \times \, \lf [ \frac{K_2 }{4\lambda}+\frac{K_1}{\sqrt{\lambda}}+2\sqrt{ 3( \lambda +K^+) }+\lf(\frac{K_1}{\sqrt{\lambda}}+2\sqrt{3(  \lambda +K^+ )}\r) \sqrt{2 \lf(1+ \log\frac{\|\phi\|_{\infty}}{\phi(x)}\r)}  \r]\\
&\leq  (\sqrt{2}+2)\sqrt{\lambda + K^+}  \lf( \frac{ (\sqrt{2}-1)K_2 }{2\sqrt{3}\lambda}+\frac{2K_1 }{\sqrt{3\lambda} }+2\sqrt{ \lambda +K^+ }\r)  \lf(1+ \log\frac{\|\phi\|_{\infty}}{\phi(x)}\r),
\end{align*}
where the second inequality is derived as 
$$ \sqrt{1+ \log\frac{\|\phi\|_{\infty}}{\phi}} \leq  1+ \log\frac{\|\phi\|_{\infty}}{\phi}.$$
Observing that the term  
\[
\sqrt{\lambda+K^+}\lf( \frac{ (\sqrt{2}-1)K_2 }{2\sqrt{3}\lambda}+\frac{2K_1 }{\sqrt{3\lambda} }\r)
\]  
is monotonically decreasing in \(\lambda\), we may replace \(\lambda\) with \(\lambda_1\) (the first eigenvalue) and use the inequality $\frac{\sqrt{2}-1}{2}\leq \frac{1}{4}$ to obtain the inequality \eqref{point-eq}.  

\end{proof}

\begin{remark}

Since the elementary inequality  
\[
\phi(x)\left(1 + \log \frac{\|\phi\|_{\infty}}{\phi(x)}\right) \leq \|\phi\|_{\infty}  
\]  
holds for all \( x \in M \), the pointwise estimate \eqref{point-eq} immediately yields the following uniform Hessian bound:  
\begin{equation}\label{uni-eq}  
\|\Hess \phi\|_{\infty} \leq (2 + \sqrt{2}) \left( \frac{K_2}{4} \sqrt{\frac{1}{3\lambda_1} + \frac{K^+}{3\lambda_1^2}} + 2K_1 \sqrt{\frac{1}{3} + \frac{K^+}{3\lambda_1}} + 2(\lambda + K^+) \right) \|\phi\|_{\infty}.  
\end{equation}  
\end{remark}

\section{Appendix}
\begin{proof}[Proof of Inequality \eqref{inequality G}]
For the case when $K$ approaches $0$, we verify that:
$$ \lim_{K\rightarrow 0}\frac{ \e^{2TK}-1-2TK \e^{TK} }{  K(\e^{KT}-1)^2} - \frac{T}{3} =0.$$ 
When $K\neq 0$, 
we need to  prove 
\begin{align*}
 \frac{ \e^{2TK}-1-2TK \e^{TK} }{  K(\e^{KT}-1)^2} - \frac{T}{3} 
&= \frac{(3-TK) \e^{2TK}-3-4TK \e^{TK}-TK}{3K(\e^{KT}-1)^2}\leq 0.
\end{align*}
This reduces to analyzing the function 
$G(x):=(3-x)\e^{2x}-3-4x\e^x-x,$
where $x = TK$, with different inequality directions depending on the sign of $K$:
\begin{align}\label{add-G}
G(x)&=
\begin{cases}
(3-x)\e^{2x}-3-4x\e^x-x\leq 0, & \text{when } K>0\,  \text{and}\,  x\geq 0;  \\
(3-x)\e^{2x}-3-4x\e^x-x\geq 0, & \text{when } K<0\,  \text{and}\,  x\leq 0. 
\end{cases}
\end{align}

Assume that  $K>0$.  Since $G(0)=G'(0)=0$,
then to prove $G(x)\leq 0$ for $x\geq 0$, it  suffices to show that $G''(x)\leq 0$ for $x\geq 0$.  Computing
second derivative of $G$:
$$G''(x)=-4\e^{2x}+4(3-x)\e^{2x}-4x\e^x-8\e^x=4\e^x\lf((2-x)\e^x-x-2 \r),$$ 
we define $H(x):=(2-x)\e^x-x-2$. Then  it is easy to verify that  $H(0)=H'(0)=0$ and for $x\geq 0$,
$$H''(x)=-x\e^{x}\leq 0,$$
 which implies that $H(x)\leq 0$ for $x>0$, and thus  $G''(x)\leq 0$ for $x\geq 0$. 
 Therefore, when $K>0$,  $G(x)\leq 0$ for $x\geq 0$.
 
 Using  a similar argument (considering $x\leq 0$)
 we can prove the second inequality in \eqref{add-G} for the case $K<0$. 
 
\end{proof}


\begin{thebibliography}{10}

\bibitem{APT2003}
Marc Arnaudon, Holger Plank, and Anton Thalmaier, \emph{A {B}ismut type formula
  for the {H}essian of heat semigroups}, C. R. Math. Acad. Sci. Paris
  \textbf{336} (2003), no.~8, 661--666. 

\bibitem{bismut1984large}
Jean-Michel Bismut, \emph{Large deviations and the {M}alliavin calculus},
  Birkhauser Prog. Math. \textbf{45} (1984).

\bibitem{chen2023bismut}
Qing-Qian Chen, Li-Juan Cheng, and Anton Thalmaier, \emph{Bismut-{S}troock
  hessian formulas and local hessian estimates for heat semigroups and harmonic
  functions on {R}iemannian manifolds}, Stochastics and Partial Differential
  Equations: Analysis and Computations \textbf{11} (2023), no.~2, 685--713.

\bibitem{CTW24}
Li-Juan Cheng, Anton Thalmaier, and Feng-Yu Wang, \emph{Second order {B}ismut
  formulae and applications to {N}eumann semigroups on manifolds}, Pure and
  Applied Functional Analysis (to appear).

\bibitem{CTW}
Li-Juan Cheng, Anton Thalmaier, and Feng-Yu Wang, \emph{Hessian estimates for
  {D}irichlet and {N}eumann eigenfunctions of {L}aplacian}, International Mathematics
  Research Notices \textbf{2024} (30 September 2024), no.~21 (English).

\bibitem{elworthy1994formulae}
K~David Elworthy and Xue-Mei Li, \emph{Formulae for the derivatives of heat
  semigroups}, Journal of Functional Analysis \textbf{125} (1994), no.~1,
  252--286.
  


\bibitem{feng2023quantitative}
Xuanrui Feng and Zhenfu Wang, \emph{Quantitative propagation of chaos for 2d
  viscous vortex model on the whole space},  arXiv:2310.05156
  (2023).

\bibitem{hamilton1993matrix}
Richard~S Hamilton, \emph{Matrix {H}arnack estimate for the heat equation},
  Communications in analysis and geometry \textbf{1} (1993), no.~1, 113--126.

\bibitem{han2016upper}
Qing Han and Qi~S Zhang, \emph{An upper bound for {H}essian matrices of positive
  solutions of heat equations}, The Journal of Geometric Analysis \textbf{26}
  (2016), 715--749.
  
     \bibitem{Ikeda-Watanabe:1989}
N. Ikeda and S. Watanabe, \emph{Stochastic differential equations and
  diffusion processes}, second ed., North-Holland Mathematical Library,
  vol.~24, North-Holland Publishing Co., Amsterdam; Kodansha, Ltd., Tokyo,
  1989. 



\bibitem{li2016hamilton}
Xiang-Dong Li, \emph{Hamilton’s {H}arnack inequality and the {W}-entropy
  formula on complete {R}iemannian manifolds}, Stochastic Processes and their
  Applications \textbf{126} (2016), no.~4, 1264--1283.


   \bibitem{Li2021}
Xue-Mei Li, \emph{Hessian formulas and estimates for parabolic
  {S}chr\"{o}dinger operators}, J. Stoch. Anal. \textbf{2} (2021), no.~3, Art.
  7, 53. 

\bibitem{li2015li}
Yi~Li, \emph{Li--{Y}au--{H}amilton estimates and {B}akry--{E}mery--{R}icci
  curvature}, Nonlinear Analysis: Theory, Methods \& Applications \textbf{113}
  (2015), 1--32.
  
 
  
    \bibitem{StT98}
Daniel~W. Stroock and James Turetsky, \emph{Upper bounds on derivatives of the
  logarithm of the heat kernel}, Comm. Anal. Geom. \textbf{6} (1998), no.~4,
  669--685. 

  
    
  \bibitem{St00}
Daniel~W. Stroock, \emph{An introduction to the analysis of paths on a
  {R}iemannian manifold}, Mathematical Surveys and Monographs, vol.~74,
  American Mathematical Society, Providence, RI, 2000. 
  
  \bibitem{Tha97}
Anton Thalmaier, \emph{On the differentiation of heat semigroups and {P}oisson
  integrals}, Stochastics Stochastics Rep. \textbf{61} (1997), no.~3-4,
  297--321.

  \bibitem{ThW98}  
 Anton Thalmaier and  Feng-Yu Wang, \emph{
Gradient Estimates for Harmonic Functions on Regular Domains in Riemannian Manifolds},
Journal of Functional Analysis, \textbf{155} (1998), no.~1, 109--124.
  
  \bibitem{Th19}
James Thompson, \emph{Derivatives of {F}eynman-{K}ac semigroups}, J. Theoret. Probab.
  \textbf{32} (2019), no.~2, 950--973.
  
\bibitem{Wbook2}
F.-Y. Wang, \emph{Analysis for diffusion processes on {R}iemannian manifolds},
  Advanced Series on Statistical Science \& Applied Probability, 18, World
  Scientific Publishing Co. Pte. Ltd., Hackensack, NJ, 2014. 


\end{thebibliography}
\providecommand{\bysame}{\leavevmode\hbox to3em{\hrulefill}\thinspace}
\providecommand{\MR}{\relax\ifhmode\unskip\space\fi MR }
\providecommand{\MRhref}[2]{%
  \href{http://www.ams.org/mathscinet-getitem?mr=#1}{#2}
}
\providecommand{\href}[2]{#2}

\end{document}